\documentclass[a4paper,11pt]{article}

\usepackage{amsfonts,qtree,stmaryrd,textcomp}
\usepackage{pifont,amssymb,amsmath,amsthm,tabularx,graphicx}
\usepackage{tree-dvips,pictexwd,dcpic}
\usepackage{bbm}
\usepackage{bm}

\usepackage{stmaryrd}

\usepackage[T1]{fontenc}
\usepackage[latin1]{inputenc}
\usepackage[text={17cm,26cm},centering]{geometry}

\usepackage{etex}
\usepackage{fancyhdr}
\usepackage{graphicx}
\usepackage{enumitem}
\usepackage{amsmath}
\usepackage[bbgreekl]{mathbbol}
\usepackage{amssymb}
\usepackage{amsthm}
\usepackage[all]{xy}

\usepackage{epigraph}
\RequirePackage{bussproofs}

\usepackage{mathabx}

\usepackage{scalerel,stackengine}

\usepackage{framed}
\usepackage{mathtools}
\usepackage{url}
\usepackage{proof}
\usepackage{xspace}
\usepackage{listings}
\usepackage{wrapfig}
\usepackage{tikz}
\usepackage{tikz-cd}

\usepackage{relsize}

\usetikzlibrary{positioning}
\usetikzlibrary{shapes}

\usetikzlibrary{arrows}
\usetikzlibrary{calc}
\usepackage{tikz-3dplot}
\usetikzlibrary{decorations.pathmorphing}

\usepackage{amsfonts,amssymb,amsmath,qtree,stmaryrd,textcomp, amsthm}

\usetikzlibrary{decorations.pathreplacing}
\tikzset{axis/.style={&lt;-&gt;}}

\stackMath
\newcommand\reallywidehat[1]{%
\savestack{\tmpbox}{\stretchto{%
  \scaleto{%
    \scalerel*[\widthof{\ensuremath{#1}}]{\kern-.6pt\bigwedge\kern-.6pt}%
    {\rule[-\textheight/2]{1ex}{\textheight}}
  }{\textheight}%
}{0.5ex}}%
\stackon[1pt]{#1}{\tmpbox}%
}
 \definecolor{MyBlue}{rgb}{0.05, 0.25, 0.65}
 
 \definecolor{MyRed}{rgb}{0.90, 0.05, 0.05}
 
\definecolor{MyGreen}{rgb}{0.05, 0.90, 0.05}

\newcommand{\B}{\boldsymbol}
\newcommand{\C}[1]{\mathcal{#1}}

\newtheorem{theorem}{Theorem}[section]
\newtheorem{proposition}[theorem]{Proposition}

\newtheorem{definition}[theorem]{Definition}

\newcommand{\Nat}{{\mathbb N}}

\newcommand{\id}{\mathrm{id}}

\newcommand{\dom}{\mathrm{dom}}

\newcommand{\Cat}{\mathrm{\mathbf{Cat}}}

\newcommand{\TOT}{\Leftrightarrow}

\newcommand{\To}{\Rightarrow}

 \newcommand{\pto}{\rightharpoondown}

\newcommand{\CM}{\mathrm{CM}}

\newcommand{\Set}{\mathrm{\mathbf{Set}}}

\newcommand{\eto}{\hookrightarrow}

\newcommand{\tro}{\rightarrowtriangle}

\newcommand{\mto}{\hookrightarrow}

\newcommand{\op}{\mathrm{op}}

\newcommand{\Hom}{\mathrm{Hom}}

\newcommand{\Asm}{\C Asm}

\newcommand{\CMt}{\textnormal{\texttt{CM}}}

\newcommand{\Sim}{\textnormal{\texttt{Siml}}}

\newcommand{\CompMod}{\B {\textnormal{\texttt{CompMod}}}}
\newcommand{\TotCompMod}{\B {\textnormal{\texttt{TotCompMod}}}}

\newcommand{\cod}{\mathrm{cod}}
\newcommand{\pull}{\mathrm{pull}}

\newcommand{\Base}{\B {\textnormal{\texttt{Base}}}}

\newcommand{\Frg}{\mathrm{Frg}}

\newcommand{\Iso}{\mathrm{Iso}}

\begin{document}

\date{}

\title{\textbf{Computability models over categories}}

\author{Iosif Petrakis\\	
Mathematics Institute, Ludwig-Maximilians-Universit\"{a}t M\"{u}nchen\\
petrakis@math.lmu.de}	
%





\maketitle

\begin{abstract}
\noindent 
Generalising slightly the notions of a strict computability model and of a simulation between them, which were 
elaborated by Longley and Normann in~\cite{LN15}, we
define canonical computability models over categories and appropriate $\Set$-valued functors on them. 
We study the canonical total computability model over a category, and the partial 
one over a 
category with pullbacks.
Our notions and results are generalised to categories with a base of computability, connecting, unexpectedly,
Rosolini's theory 
of dominions with the theory of computability models. 
\\[2mm]
\end{abstract}

\section{Introduction}
\label{sec: intro}

In their book~\cite{LN15}
Longley and Normann not only give a comprehensive introduction to Higher-Order Computability
(HOC) presenting the various approaches to HOC, but they also fit them into a coherent and unifying
framework, making their comparison possible. Their main tool is a
general notion of computability model\footnote{This notion
is rooted in previous work of Longley in~\cite{Lo95}-\cite{Lo14}, and is influenced by the work of Cockett and Hofstra
in~\cite{CH08} and~\cite{CH10} (see~\cite{LN15}, p.~52).} and of an appropriate concept of simulation of one 
computability model in another. Turing machines, programming languages,
$\lambda$-calculus etc., are shown to be computational models, not necessarily in a unique way.

As it is remarked in~\cite{Co14}, p.~3, computability theory has ``still not received the level of categorical 
attention it deserves''. It seems though, that the ``categorical spirit'' of the notions of a computability model and a 
simulation between them is behind the remarkable success of the framework of Longley and Normann.
E.g., their, completely categorical in nature,
notion of equivalence of computability models (see Definition~\ref{def: siml}) is deeper
than the one suggested by the standard presentations of computability theory, according to which different notions 
of computability are equivalent if they generate the same class of partial computable functions from $\Nat$ to $\Nat$
(see~\cite{LN15}, Section 1.1.5). Longley and Normann associated in a canonical way 
to a computability model\footnote{In~\cite{LN15} 
this category is defined for lax computability models, but, as it is remarked in~\cite{LN15}, p.~91, the definition makes
sense for an arbitrary computability model.} $\B C$ its category of
assemblies $\Asm(\B C)$, and showed that the computability models $\B C$ and $\B D$ are 
equivalent if and only if the categories of assemblies $\Asm(\B C)$ and $\Asm (\B D)$ are
equivalent. The category of computability models and the corresponding functor $\B C \mapsto \Asm(\B C)$ 
are studied extensively in~\cite{Lo14}.

In this paper we associate in a canonical way a computability model to a category $\C C$, given an appropriate 
$\Set$-valued functor on $\C C$. For that we slightly generalise the definition of a computability model, allowing 
the classes of type names $T$ and of partial functions $C[\sigma, \tau]$ in a computability model to be proper classes.
The computability models considered in~\cite{LN15} are called \textit{small} in Definition~\ref{def: cm}, and 
working solely with them we can only define the computability model of a small category. 
The notion of a simulation between computability 
models is also slightly generalised, as the function between the corresponding classes of type names is, generally, 
a class-function. Although Longley and Normann usually work with \textit{lax} computability models, to our needs 
the so-called \textit{strict} computability models suit 
best\footnote{For a discussion on ``strict vs. lax'' see \cite{Lo14}, section 2.1.}.

In section~\ref{sec: cmodels} we include all basic notions and facts necessary to the rest of this paper.
In section~\ref{sec: totalcm} we study the canonical total computability model over a category $\C C$ and 
a functor $S \colon \C C \to \Set$. In section~\ref{sec: partialcm} we study the canonical (partial) 
computability model over a category $\C C$ with pullbacks and 
a functor $S \colon \C C \to \Set$ that preserves pullbacks. In section~\ref{sec: base} we define the notion of 
a base of computability for a category $\C C$, and the canonical (partial) 
computability model over a category $\C C$ with a base of computability for $\C C$ and a functor $S \colon \C C \to \Set$ 
that preserves pullbacks. The first two constructions are shown to be special cases of the last one. All computability models
defined here have their dual form e.g., the computability model over a category $\C C$ with a cobase of computability 
for $\C C$ and a functor $S \colon \C C \to \Set$ 
that sends pushouts to pullbacks is also defined.

For all categorical notions and facts mentioned here we refer to~\cite{Aw10}.


\section{Computability models}
\label{sec: cmodels}



\begin{definition}\label{def: cm}
Let $T$ be a class, the elements of which are called type names. 
A computability model $\B C$ over  $T$ is a pair 
$$\B C = \big(\big(\B C(\tau)\big)_{\tau \in T},\big(\B C[\sigma, \tau]\big)_{(\sigma, \tau) \in T \times T}\big),$$ 
where $|\B C| = \big(\B C(\tau)\big)_{\tau \in T}$ is a family of sets $\B C(\tau)$ over $T$, called the
datatypes of $\B C$, and $\big(\B C[\sigma, \tau]\big)_{(\sigma, \tau) \in T \times T}$ is a family of 
classes $\B C[\sigma, \tau]$  of \textit{partial functions} 
$\B C(\sigma) \pto \B C(\tau)$ over $T \times T$, such that the following hold:\\[1mm]
$(\CMt_1)$ The identity function $\id_{\B C(\tau)} \colon \B C(\tau) \to \B C(\tau)$ is in $\B C[\tau, \tau]$, for 
every $\tau \in T$,\\[1mm]
$(\CMt_2)$ For every $f \in \B C[\rho, \sigma]$ and $g \in \B C[\sigma, \tau]$, the composite partial
function $g \circ f$ is in 
$\B C[\rho, \tau]$, for every $\rho, \sigma, \tau \in T$.\\[1mm]
If the classes $\B C[\sigma, \tau]$ are sets, for every $\sigma, \tau \in T$, we call $\B C$ locally small. 
If $T$ is also a set, we call $\B C$ small.
If every element of $\B C[\sigma, \tau]$ is a total function, for every $\sigma, \tau \in T$, then $\B C$ is called total.
\end{definition}

For the rest of this paper $T, U, W$  are classes, and $\B C, \B D, \B E$ are computability models over $T, U, W$, 
respectively. 

\begin{definition}\label{def: siml}
A \textit{simulation} $\B \gamma$ of $\B C$ in $\B D$, in symbols $\B \gamma \colon \B C \tro \B D$, is a pair
$$\B \gamma = \big(\gamma, \big(\Vdash^{\B \gamma}_{\tau}\big)_{\tau \in T}\big),$$
where $\gamma \colon T \to U$ is a class function, and $\Vdash^{\B \gamma}_{\tau} \subseteq \B D(\gamma(\tau))
\times \B C(\tau)$, for every $\tau \in T$,  such that the following hold:
$$(\Sim_1) \ \ \ \ \ \ \ \ \ \ \ \ \ \ \ \ \ \ \ \ \ \ \ \ \ \ \ \ \forall_{\tau \in T}\forall_{x \in 
\B C(\tau)}\exists_{x{'} \in \B D(\gamma(\tau))}\big(x{'} \Vdash^{\B \gamma}_{\tau} x \big), \ \ \ \ \ 
\ \ \ \ \ \ \ \ \ \ \ \ \ \ \ \ \ \ \ \ \ \ \ \ $$
$$(\Sim_2) \ \ \ \ \ \ \ \ \ \ \ \ \ \ \ \ \ \ \ \ \ \ \ \ \ \ \ \ \ \forall_{\sigma, \tau
\in T}\forall_{f \in \B C[\sigma, \tau]}\exists_{f{'} \in \B D[\gamma(\sigma), \gamma(\tau)]}\big(f{'} 
\Vdash^{\B \gamma}_{(\sigma, \tau)} f\big), \ \ \ \ \ \ \ \ \ \ \ \ \ \ \ \ \ \ \ \ \ \ \ $$
where the relation ``$f$ is tracked by $f{'}$ through $\B \gamma$'' is defined by
$$f{'} \Vdash^{\B \gamma}_{(\sigma, \tau)} f :\TOT \forall_{x \in \B C(\sigma)}\big(x \in \dom(f) \To $$
$$\forall_{x{'} \in \B D(\gamma(\sigma))}\big(x{'} \Vdash^{\B \gamma}_{\sigma} x \To x{'} \in \dom(f{'}) \ 
\& \ f{'}(x{'}) \Vdash^{\B \gamma}_{\tau} f(x)\big)\big).$$
The identity simulation $\B \iota_{\B C} \colon \B C \tro \B C$ is the pair 
$\big(\id_T, (\Vdash^{\B \iota_{\B C}}_{\tau})_{\tau \in T}\big)$, where $x{'} \Vdash^{\B \iota_{\B C}}_{\tau} \TOT x{'} = x$, for every $x{'}, x \in \B C(\tau)$. If $\B \delta \colon \B D \tro \B E$, the composite simulation $\B \delta \circ \B \gamma \colon \B C \tro \B E$ is the pair $\big(\delta \circ \gamma, (\Vdash^{\B \delta \circ \B \gamma}_{\tau})_{\tau \in T}\big)$, where the relation $\Vdash^{\B \delta \circ \B \gamma}_{\tau} \subseteq \B E\big(\delta(\gamma(\tau))\big) \times \B C(\tau)$ is defined by
$$z \Vdash^{\B \delta \circ \B \gamma}_{\tau} x \TOT \exists_{y \in \B D(\gamma(\tau))}\big(z 
\Vdash^{\B \delta}_{\gamma(\tau)} y \ \& \ y \Vdash^{\B \gamma}_{\tau} x\big).$$
We denote by $\CompMod$ the category of 
computability models with simulations. If $\B \gamma, \B \delta \colon \B C \tro \B D$, then $\B \gamma$ 
is \textit{transformable} to $\B \delta$, in symbols $\B \gamma \preceq \B \delta$, if
$$\forall_{\tau \in T}\exists_{f \in \B D[\gamma(\tau), \delta(\tau)]}\forall_{x \in 
\B C(\tau)}\forall_{x{'}\in \B D(\gamma(\tau))}\big(x{'} \Vdash^{\gamma}_{\tau}x \To x{'} \in \dom(f) \ 
\& \ f(x{'}) \Vdash^{\B \delta}_{\tau} x\big).$$
The simulations $\B \gamma, \B \delta$ are \textit{equivalent}, in symbols $\B \gamma \sim \B \delta$, if
$\B \gamma \preceq \B \delta$ and $\B \delta \preceq \B \gamma$. The computational models $\B C$ and $\B D$ 
are \textit{equivalent}, if there are simulations $\B \gamma \colon \B C \tro \B D$ and $\B \delta \colon \B D 
\tro \B C$ such that $\B \delta \circ \B \gamma \sim \B {\iota_{\B C}}$ and $\B \gamma \circ \B \delta \sim
\B {\iota_{\B D}}$.
\end{definition}

If $f{'} \Vdash^{\B \gamma}_{(\rho, \sigma)} f$ and $g{'} \Vdash^{\B \gamma}_{(\sigma, \tau)} g$, 
then $g{'} \circ f{'} \Vdash^{\B \gamma}_{(\rho, \tau)} g \circ f$. Similarly, if $\B \gamma \preceq \B {\gamma{'}}$
and $\B \delta \preceq \B {\delta{'}}$, then $\B \delta \circ \B \gamma \preceq \B {\delta{'}} \circ \B {\gamma{'}}$, 
where $\B \gamma, \B {\gamma{'}} \colon \B C \tro \B D$ and $\B \delta, \B {\delta{'}} \colon \B D \tro \B E$.
%

\begin{definition}\label{def: assemblies}
The category $\Asm (\B C)$ of assemblies over 
$\B C$ has objects triplets $(X, \tau_X,
\Vdash_X)$, where $X$ is a set, $\tau_X \in T$, and $\Vdash_X \subseteq \B C(\tau_X) \times X$ such that
$$\forall_{x \in X}\exists_{x{'} \in \B C(\tau_X)}\big(x{'} \Vdash_X x\big).$$
A morphism $(X, \tau_X, \Vdash_X)
\to (Y, \tau_Y, \Vdash_Y)$ is a function $f \colon X \to Y$, such that there is $\bar{f} \in \B C[\tau_X, \tau_Y]$ 
that ``tracks'' $f$, in symbols $\bar{f} \Vdash^Y_X f$, where 
$$\bar{f} \Vdash^Y_X f :\TOT \forall_{x \in X}\forall_{y \in \B C(\tau_X)}\big(y \Vdash_X x  \To 
y \in \dom(\bar{f}) \ \& \ \bar{f}(y) \Vdash_Y f(x)\big).$$
The forgetful functor $\Frg^{\B C} \colon \Asm (\B C) \to \Set$ is defined, as usual, by 
$\Frg^{\B C}_0(X, \tau_X, \Vdash_X) = X$ and $\Frg^{\B C}_1(f) = f$. 
\end{definition}

Clearly, the forgetful functor $\Frg^{\B C}$ is injective on arrows.


%
%
%
%
%

\section{Total computability models over categories}
\label{sec: totalcm}

For the rest of this paper $\C C, \C D$  are categories, $\Cat$ is the category of categories, $\Set$ is the category of
sets, $[\C C, \Set]$ is the category of functors from $\C C$ to $\Set$, and $\Cat/\Set$ is the slice category of 
$\Cat$ over $\Set$. Let $C_0, C_1$ be the classes of objects and morphisms in $\C C$, respectively.
If $a, b \in C_0$, let $C_1(a, b)$ be the class of morphisms $f$ in $C_1$
with $\dom(f) = a$ and $\cod(f) = b$.

\begin{definition}\label{def: total}
Let $S \colon \C C \to \Set$ a $($covariant$)$ functor. The total computability model 
$\B {\CM}^t(\C C ; S)$ over $\C C$ and $S$ is the pair
$$\B {\CM}^t(\C C ; S) = \big(\big(S_0(a)\big)_{a \in C_0}, \big(S[a, b]\big)_{(a, b) \in C_0 \times C_0}\big),$$ 
$$S[a, b] = \{S_1(f) \mid f \in C_1(a, b)\}.$$
If $T \colon \C C^{\op} \to \Set$ is contravariant, the total computability model 
$\B {\CM}^t(\C C ; T)$ over $\C C$ and $T$ is defined dually i.e., $T[a, b] = \{T_1(f) \mid f \in C_1(b, a)\}.$
\end{definition}

Clearly, $\B {\CM}^t(\C C ; S)$ satisfies conditions $(\CM_1)$ and $(\CM_2)$.
With replacement, if $\C C$ is locally small, then $\B {\CM}^t(\C C ; S)$ is locally small, and if $\C C$ is small, then 
$\B {\CM}^t(\C C ; S)$ is small.
Next we describe an ``external'' and an ``internal''
functor in the full subcategory $\TotCompMod$ of total computability models. The corresponding proofs are omitted, as these 
are similar to the included proofs of Propositions~\ref{prp: partial1} and~\ref{prp: partial2}.

\begin{proposition}\label{prp: total1}
There is a functor $\CM^t \colon \Cat/\Set \to \TotCompMod$, defined by
$$\CM^t_0(\C C, S) = \B \CM^t(\C C ; S),$$
$$\CM^t_1\big(F \colon (\C C, S) \to (\C D, T)\big) : \B \CM^t(C ; S) \tro \B \CM^t(\C D ; T),$$
$$\CM^t_1(F) = \B \gamma_F = \big(F_0, (\Vdash^{\B \gamma_F}_a)_{a \in C_0}\big),$$
$$\Vdash^{\B \gamma_F}_a \subseteq T_0(F_0(a)) \times S_0(a), \ \ \ \ y \Vdash^{\B \gamma_F}_a x \TOT y = x.$$
\end{proposition}


\begin{proposition}\label{prp: total2}
There is a faithful functor
$^t\CM^{\C C} \colon [\C C, \Set] \to \TotCompMod$, defined by
$$^t\CM_0^{\C C}(S) = \B \CM^t(\C C ; S),$$
$$^t\CM_1^{\C C}\big(\eta \colon S \To T\big) : \B \CM^t(C ; S) \tro\B  \CM^t(\C D ; T),$$
$$^t\CM_1^{\C C}(\eta) = \B \gamma_{\eta} = \big(\id_{C_0}, (\Vdash^{\B \gamma_{\eta}}_a)_{a \in C_0}\big),$$
$$\Vdash^{\B \gamma_{\eta}}_a \subseteq T_0(a) \times S_0(a), \ \ \ \ y \Vdash^{\B \gamma_{\eta}}_a x \TOT y = \eta_a(x).$$
\end{proposition}


By Definition~\ref{def: assemblies} 
the category $\Asm\big(\B {\CM}^t(\C C;S)\big)$ of assemblies of the total
computability model $\B {\CM}^t(\C C ; S)$ 
has objects triplets $(X, a_X,
\Vdash_X)$, where $X$ is a set, $a_X \in C_0$, and $\Vdash_X \subseteq S_0(a_X) \times X$ such that
$\forall_{x \in X}\exists_{x{'} \in S_0(a_X)}\big(x{'} \Vdash_X x\big)$. A morphism $(X, a_X, \Vdash_X)
\to (Y, a_Y, \Vdash_Y)$ is a function $f \colon X \to Y$, such that there is $\bar{f} \in S[a_X, a_Y]$
i.e., $g \colon a_X \to a_Y \in C_1$ and $\overline{f} = S_1(g) \colon S_0(a_X) \to S_0(a_Y)$
tracks $f$, where in this case 
$\bar{f} \Vdash^Y_X f$ if and only if $\forall_{x \in X}\forall_{y \in S_0(a_X)}\big(y \Vdash_X x  \To 
[S_1(g)](y) \Vdash_Y f(x)\big)$.

\begin{proposition}\label{prp: total3}
Let $F^t \colon \C C \to \Asm\big(\B {\CM}^t(\C C;S)\big)$ be defined by
$$F^t_0(a) = \big(S_0(a), a, \Vdash_{S_0(a)}\big); \ \ \ \ a \in C_0,$$
$$y \Vdash_{S_0(a)} x \TOT y = x; \ \ \ \ y, x \in S_0(a),$$
$$F^t_1(f \colon a \to b) \colon \big(S_0(a), a \Vdash_{S_0(a)}\big) \to \big(S_0(b), b \Vdash_{S_0(b)}\big), \ \ 
\ \ F^t_1(f) = S_1(f).$$
\normalfont (i) 
\itshape $F^t$ is a full functor, injective on objects.\\[1mm]
\normalfont (ii) 
\itshape If $S$ is injective on arrows, then $F^t$ is an embedding. 
\end{proposition}

\begin{proof}
(i) Clearly, $S_1(f)$ tracks itself, and hence $S_1(f)$ is a morphism in 
$\Asm\big(\B {\CM}^t(\C C;S)\big)$. Clearly, $F^t$ is a functor, injective on objects. Let $g \colon 
\big(S_0(a), a \Vdash_{S_0(a)}\big) \to \big(S_0(b), b \Vdash_{S_0(b)}\big)$ a morphism in 
$\Asm\big(\B {\CM}^t(\C C;S)\big)$ i.e., a function $g \colon S_0(a) \to S_0(b)$ such that there is 
$\overline{g} = S_1(h)$, for some $h \colon a \to b$ in $C_1$, with $S_1(h) \Vdash^{S_0(b)}_{S_0(a)} g$ i.e.,
$\forall_{x \in S_0(a)}\forall_{y \in S_0(a)}\big(y = x  \To [S_1(h)](y) = g(x)\big)$, hence
$\forall_{x \in S_0(a)}\big([S_1(h)](x) = g(x)\big)$ i.e., $F_1^+(h) = S_1(h) = g$ and $F^t$ is full.\\ 
(ii)  If $S$ is injective on arrows, then $F^t$ is also injective on arrows, hence faithful, and since it is also 
full and injective on objects, it is an embedding.
\end{proof}

Let $\B C$ be a small, and $\Frg^{\B C}$ 
the forgetful functor on $\Asm(\B C)$.
If $\overline{X} = (X, \tau_X, \Vdash_X)$ and
$\overline{Y} = (Y, \tau_Y, \Vdash_Y)$, the 
computability model $\B {\CM}^t(\Asm(\B C) ; \Frg^{\B C})$ 
over $\Asm(\B C)$ and $\Frg^{\B C}$ is the pair
$$\B {\CM}^t(\Asm(\B C) ; \Frg^{\B C}) = \bigg(\big(X\big)_{\overline{X} \in \Asm(\B C)_0}, 
\big(\Frg^{\B C}[\overline{X}, \overline{Y}]\big)_{\overline{X}, \overline{Y} \in \Asm(\B C)_0}\bigg),$$ 
$$\Frg^{\B C}[\overline{X}, \overline{Y}] = \big\{f \colon X \to Y \mid
\exists_{\overline{f} \colon C(\tau_x) \pto C(\tau_Y)}(\overline{f} \Vdash^Y_X f)\big\}.$$
I.e., from a small, and, in general, partial computability model $\B C$, we get the locally small and total
computability model $\B {\CM}^t\big(\Asm(\B C) ; \Frg^{\B C}\big)$. Due to lack of space, the proof of 
the following proposition is omitted.

\begin{proposition}\label{prp: total4}
\normalfont (i)  
\itshape If $\B C$ is small, there is a simulation 
$$\B \delta^t \colon \B {\CM}^t\big(\Asm(\B C) ; \Frg^{\B C}\big) \tro \B C,$$ 
where
$\B \delta^t = \big(\delta^t,\big(\Vdash^{\B \delta^t}_{\overline{X}}\big)_{\overline{X} \in \Asm(\B C)_0}\big),$
with $\delta^t \colon \Asm(\B C)_0 \to T$, defined by $\delta^t(\overline{X}) = \tau_X$, and  
$\Vdash^{\B \delta^t}_{\overline{X}} \subseteq C(\tau_X) \times X$ defined by $\Vdash^{\B \delta^t}_{\overline{X}} \ = \
\Vdash_X$.\\[1mm]
\normalfont (ii)
\itshape If $\B C$ is total and small, there is a simulation $$\B \gamma^t \colon \B C \tro 
\B {\CM}^t\big(\Asm(\B C) ; \Frg^{\B C}\big),$$
where
$\B \gamma^t = \big(\gamma^t,\big(\Vdash^{\B \gamma^t}_\tau\big)_{\tau \in T}\big),$
with $\gamma^t \colon T \to \Asm(\B C)_0$, defined by $\gamma^t(\tau) = (C(\tau), \tau, \Vdash_{C(\tau)})$
and $\Vdash_{C(\tau)} \subseteq C(\tau) \times C(\tau)$ is the equality on $C(\tau)$, and  
$\Vdash^{\B \gamma^t}_{\tau} \subseteq C(\tau) \times C(\tau)$ is also the equality on $C(\tau)$.\\[1mm]
\normalfont (iii)
\itshape If $\B C$ is total and small, then $\B \delta^t \circ \B \gamma^t = \B \iota_{\B C}$.\\[1mm]
\end{proposition}

If $\B C$ is total and small, then $\B \gamma^t \circ \B \delta^t \sim
\B \iota_{\B {\CM}^t(\Asm(\B C) ; \Frg^{\B C})}$.
Similarly, if $\C A^t = \Asm\big(\B {\CM}^t(\C C;S)\big)$ and $\Frg^{\C A^t}$ is
the forgetful functor on $\C A^t$, 
the 
computability 
models ${\CM}^t(\C C;S)$ and ${\CM}^t\big(\C A^t, \Frg^{\C A^t}\big)$ are expected to be equivalent.

\section{Computability models over categories with pullbacks}
\label{sec: partialcm}

\begin{definition}\label{def: partial}
Let $C_1(a \eto b)$ be the subclass of monos in $C_1(a, b)$, let
$C_1(a \eto)$ be the subclass of monos in $C_1$ with domain $a$, and $C_1(\eto a)$ the subclass of monos in $C_1$ 
with codomain $a$. If $\C C$ has pullbacks, a partial
morphism\footnote{As here we don't define categories of partial morphisms, 
we avoid the equivalence relation between them. A
partial function between sets is a partial morphism in $\Set$. 
This is in accordance with the notion of partial function between sets in Bishop set theory see~\cite{Pe19} 
and~\cite{Pe20}.}
$(i, f) \colon a \pto b$ is a pair $(i, f)$, where $i \in C_1(\eto a)$
and $f \in C_1(\dom(i), b)$

%

\begin{center}
\begin{tikzpicture}

\node (E) at (0,0) {$s$};
\node[right=of E] (F) {$a$};
\node[below=of F] (A) {$b$};
\node[right=of F] (K) {$S_0(s)$};
\node[right=of K] (L) {$S_0(a)$};
\node[below=of L] (M) {$S_0(b)$.};

\draw[right hook ->] (E)--(F) node [midway,above] {$i$};
\draw[->] (E) to node [midway,left] {$ f \ $} (A) ;
\draw[right hook ->] (K)--(L) node [midway,above] {$S_1(i)$};
\draw[->] (K) to node [midway,left] {$ S_1(f) \ $} (M) ;
\draw[blue,-left to] (F) to node [midway,right] {$  (i, f)$} (A) ;
\draw[blue,-left to] (L) to node [midway,right] {$  S_1(i, f)$} (M) ;

\end{tikzpicture}
\end{center}
If $(j, g) \colon b \pto c$, where $j \colon t \mto b$ and $g \colon t \to c$, their composition is the 
partial morphism $(i \circ i{'}, g \circ f{'}) \colon a \pto c$, where $i{'} \colon s \times_b t \mto s$ 
and $f{'}  \colon s \times_b t \to t$ is determined by the corresponding pullback diagram
\begin{center}
\begin{tikzpicture}

\node (E) at (0,0) {$s$};
\node[right=of E] (F) {$a$};
\node[below=of F] (A) {$ b$};
\node[left=of E] (P) {$s \times_b t$};
\node[left=of A] (T) {$t \ \ $};
\node[below=of A] (C) {$c$};

\draw[right hook ->] (E)--(F) node [midway,above] {$i$};
\draw[->] (E) to node [midway,right] {$ \ f $} (A) ;
\draw[->] (T) to node [midway,left] {$g \ $} (C) ;
\draw[->] (P) to node [midway,left] {$f{'} \ $} (T) ;
\draw[right hook ->] (T)--(A) node [midway,above] {$j$};
\draw[right hook ->] (P)--(E) node [midway,above] {$i{'}$};
\draw[blue,-left to] (F) to node [midway,right] {$(i, f)$} (A) ;
\draw[blue,-left to] (A) to node [midway,right] {$(j, g)$} (C) ;
\draw[blue,-left to, bend left=90] (F) to node [midway,right] {$  (i \circ i{'}, g \circ f{'})$} (C) ;

\end{tikzpicture}
\end{center}
If $S \colon \C C \to \Set$ is a functor that preserves pullbacks, hence monos, and if $(i, f) \colon a \pto b$,
let
$S_1(i, f) = (S_1(i), S_1(f)) \colon S_0(a) \pto S_0(b)$. Let 
$$\dom(S_1(i, f)) = \big\{[S_1(i)](x) \mid x \in S_0(s)\big\},$$
$$[S_1(i, f)](y) = [S_1(f)](x); \ \ \ \ y = [S_1(i)](x) \in \dom(S_1(i, f)).$$
\end{definition}

\begin{definition}\label{def: partialcomp}
If $\C C$ has pullbacks, and $S \colon \C C \to \Set$  
preserves pullbacks, the $($partial$)$ computability model $\B {\CM}^p(\C C ; S)$ over $\C C$
and $S$ is the pair
$$\B {\CM}^p(\C C ; S) = \big(\big(S_0(a)\big)_{a \in C_0}, \big(S[a, b]\big)_{(a, b) \in C_0 \times C_0}\big),$$
$$S[a, b] = \{S_1(i, f) \mid i \in C_1(\eto a) \ \& \ f \in C_1(\dom(i), b)\}.$$
If $\C D$ has pushouts, and $T \colon \C C^{\op} \to \Set$ sends 
pushouts to pullbacks, the $($partial$)$ computability model $\B {\CM}^p(\C D ; T)$ over $\C D$
and $T$ is defined dually i.e., $T[a, b] = \{T_1(i, f) \mid i \in C_1(a \eto) \ \& \ f \in C_1(b, \cod(i))\}$.
\end{definition}

The fact that $\B {\CM}^p(\C C ; S)$ is a computability model depends crucially on the preservation of pullbacks by 
$S$. If $\C C$ is small, then $\B {\CM}^p(\C C ; S)$ is small, but if $\C C$ is locally small, 
$\B {\CM}^p(\C C ; S)$ need not be locally small. 
Next we describe the corresponding ``external'' and ``internal''
functors.

\begin{proposition}\label{prp: partial1}
Let $\Cat^{\pull}$ be the category of categories with pullbacks with morphisms the functors preserving 
pullbacks, and let $\Cat^{\pull}/\Set$ the slice category over $\Set$. If $\Cat^{\pull}_{\eto}{/}\Set$ is the subcategory of 
$\Cat^{\pull}/\Set$ with morphisms the morphisms in $\Cat^{\pull}/\Set$ that preserve monos\footnote{If 
$F \colon \C C \to \C D$ such that $T \circ F = S$, where $S \colon \C C \to \Set$ and $T \colon \C D \to \Set$, we cannot
show, in general, that $F$ preserves monos. It does, if, e.g., $T$ is injective on arrows.}, then  
%
%
$$\CM^s \colon \Cat^{\pull}_{\eto}/\Set \to \CompMod$$
is a functor, where
$\CM^p_0(\C C, S) = \B \CM^p(\C C ; S),$ and
$\CM^p_1(F) : \B \CM^p(C ; S) \tro \B \CM^p(\C D ; T)$ is defined as the simulation $\CM_1^t(F) = \B \gamma_{F}$
in Proposition~\ref{prp: total1}.
\end{proposition}

\begin{proof}
First we show that $\B \gamma_{F} = \big(F_0, (\Vdash^{\B \gamma_{F}}_a)_{a \in C_0}\big),$
where $\Vdash^{\B \gamma_{F}}_a \subseteq T_0(F_0(a)) \times S_0(a)$ is defined by $y \Vdash^{\B \gamma_{F}}_a x \TOT 
y = x$, is a simulation $\B \CM^p(C ; S) \tro\B  \CM^p(\C D ; T)$. To show $(\Sim_1)$, if $a \in C_0$ and 
$x \in S_0(a)$, then $S_0(a) = T_0(F_0(a)) \ni x = \Vdash^{\B \gamma_{F}}_a x$. To show $(\Sim_2)$, if $a, b \in C_0$
and $(S_1(i), S_1(f)) \colon S_0(a) \pto S_0(b) \in S[a, b]$, for some $s \in C_0$, we find an element of 
$T[F_0(a), F_0(b)]$ that tracks $(S_1(i), S_1(f))$. Let $(T_1 \circ F_1)(i, f) = \big(T_1(F_1(i)), T_1(F_1(f))\big)
\colon T_0(F_0(a)) \pto T_0(F_0(b))$, as $F$ preserves monos. If $x \in \dom(S_1(i, f))$ i.e., 
there is a (unique) $x{'} \in S_0(s)$ with $x = [S_1(i)](x{'})$, then if $y \in T_0(F_0(a) = S_0(a)$ such that $y
\Vdash^{\B \gamma_{F}}_a x$, then $y = x$. We show that $x \in \dom\big((T_1 \circ F_1)(i, f)\big)$, which 
by Definition~\ref{def: partial} means that $x = [T_1(F_1(i))](x{''})$, for some $x{'} \in T_0(F_0(s)) = S_0(s)$. Clearly,
$x{''} = x{'}$ works. By Definition~\ref{def: partial} the equality obtained unfolding the relation 
$[(T_1 \circ F_1(i, f)](x) \Vdash^{\B \gamma_{F}}_b [S_1(i, ff)](x)$ follows immediately. Clearly,
$\CM_1^p(1_{(\C C, S)}) = \B \iota_{\B \CM^p(\C C ; S)}$, and if $F \colon \C C \to \C D$,
$G \colon \C D \To \C E$, then $\B \gamma_{G \circ F} = \B \gamma_{G} \circ \B \gamma_{F}$.
\end{proof}

\begin{proposition}\label{prp: partial2}
If $\C C$ has pullbacks and $[\C C, \Set]^{\pull}$ is the category of functors from
$\C C$ to $\Set$ preserving pullbacks, then 
$$^p\CM^{\C C} \colon [\C C, \Set]^{\pull} \to \CompMod$$
is a faithful functor, 
where $^p\CM_0^{\C C}(S) = \B \CM^p(\C C ; S),$ and the simulation
$^p\CM_1^{\C C}(\eta) : \B \CM^p(C ; S) \tro\B  \CM^p(\C C ; T)$ is defined as $^t\CM_1^{\C C}(\eta) = \B \gamma_{\eta}$ 
in Proposition~\ref{prp: total2}.
\end{proposition}

\begin{proof}
First we show that $\B \gamma_{\eta} = \big(\id_{C_0}, (\Vdash^{\B \gamma_{\eta}}_a)_{a \in C_0}\big),$
where $\Vdash^{\B \gamma_{\eta}}_a \subseteq T_0(a) \times S_0(a)$ is defined by $y \Vdash^{\B \gamma_{\eta}}_a x \TOT 
y = \eta_a(x)$, is a simulation $\B \CM^p(C ; S) \tro\B  \CM^p(\C C ; T)$. To show $(\Sim_1)$, if $a \in C_0$ and 
$x \in S_0(a)$, then $T_0(a) \ni y = \eta_a(x) \Vdash^{\B \gamma_{\eta}}_a x$. To show $(\Sim_2)$, if $a, b \in C_0$
and $(S_1(i), S_1(f)) \colon S_0(a) \pto S_0(b) \in S[a, b]$, for some $s \in C_0$, we find an element of $T[a, b]$ 
that tracks $(S_1(i), S_1(f))$. Let $T_1(i, f) = \big(T_1(i), T_1(f)\big) \colon T_0(a) \pto T_0(b)$. We show that 
$$\big(T_1(i), T_1(f)\big) \Vdash^{\B \gamma_{\eta}}_{(a, b)} \big(S_1(i), S_1(f)\big).$$
Let $x \in S_0(a)$ such that $x \in \dom(S_1(i, f))$ i.e., there is a (unique) $x{'} \in S_0(s)$ with $x = [S_1(i)](x{'})$.
Let $y \in T_0(a)$ such that $y \Vdash^{\B \gamma_{\eta}}_a x \TOT y = \eta_a(x)$. First we show that 
$y \in \dom(T_1(i, f))$ i.e., $y = [T_1(i)](y{'})$, for some $y{'} \in T_0(s)$. If $y{'} = \eta_s(x{'})$, 
then by the commutativity of the following left
diagram we have that
\begin{center}
\begin{tikzpicture}

\node (E) at (0,0) {$S_0(s)$};
\node[right=of E] (F) {$S_0(a)$};
\node[below=of E] (S) {$T_0(s)$};
\node[below=of F] (A) {$T_0(a)$};

\node[right=of F] (K) {$S_0(s)$};
\node[right=of A] (L) {$T_0(s)$};
\node[right=of K] (M) {$S_0(b)$};
\node[right=of L] (N) {$T_0(b)$};

\draw[->] (F)--(A) node [midway,right] {$\eta_a$};
\draw[->] (E) to node [midway,above] {$S_1(i) $} (F) ;
\draw[->] (E)--(S) node [midway,left] {$\eta_s$};
\draw[->] (S) to node [midway,below] {$T_1(i)$ } (A) ;
\draw[->] (M)--(N) node [midway,right] {$\eta_b$};
\draw[->] (K) to node [midway,above] {$S_1(f) $} (M) ;
\draw[->] (K)--(L) node [midway,left] {$\eta_s$};
\draw[->] (L) to node [midway,below] {$T_1(f)$ } (N) ;

\end{tikzpicture}
\end{center}
$$[T_1(i)](y{'}) = [T_1(i)](\eta_s(x{'})) = \eta_a\big([S_1(i)](x{'})\big) = \eta_a(x) = y.$$
Next we show that $[T_1(i, f)](y) \Vdash^{\B \gamma_{\eta}}_b [S_1(i, f)](x)$ i.e., 
$[T_1(f)](y{'}) \Vdash^{\B \gamma_{\eta}}_b [S_1(f)](x{'})$. By the commutativity of the above right diagram
we have that 
$$\eta_b\big([S_1(f)](x{'})\big) = [T_1(f)]\big(\eta_s(x{'})\big) = [T_1(f)](y{'}).$$
It is straightforward to show that $^p\CM_1(1_S \colon S \To S) = \B \iota_{\B \CM^p(\C C ; S)}$, 
and if $\tau \colon S \To T$,
$\eta \colon T \To U$, then $\B \gamma_{\eta \circ \tau} = \B \gamma_{\eta} \circ \B \gamma_{\tau}$.
To show that $^p\CM^{\C C}$ is faithful, let $\eta, \theta \colon S \To T$ such that 
$\B \gamma_{\eta} = \B \gamma_{\theta}$ i.e.,
$\Vdash^{\B \gamma_{\eta}}_a = \Vdash^{\B \gamma_{\theta}}_a$, for every $a \in C_0$. If we fix some $a \in C_0$, let
$y \in T_0(a)$ and $x \in S_0(a)$. Then $y \Vdash^{\B \gamma_{\eta}}_a x \TOT y \Vdash^{\B \gamma_{\theta}}_a x$ i.e.,
$y = \eta_a(x) \TOT y = \theta_a(x)$, hence $\eta_a = \theta_a$, and since $a \in C_0$ is arbitrary, we conclude that
$\eta = \theta$.
\end{proof}

The category $\Asm\big(\B {\CM}^p(\C C;S)\big)$ of assemblies of 
$\B {\CM}^p(\C C ; S)$ 
has objects the standard triplets $(X, a_X,
\Vdash_X)$, and a morphism $(X, a_X, \Vdash_X)
\to (Y, a_Y, \Vdash_Y)$ is a function $f \colon X \to Y$, such that there is $\bar{f} \in S[a_X, a_Y]$
i.e., there is $s \in C_0$ and $i \colon s \eto a_X, f \colon s \to a_Y$ such that $\bar{f}  = \big(S_1(i), S_1(f)\big)$
and $\bar{f} \Vdash^Y_X f$. One can define a functor $F^p \colon \C C \to \Asm\big(\B {\CM}^p(\C C;S)\big)$, as in 
Proposition~\ref{prp: total3}. In this case the function $F^p_1(f \colon a \to b) =  S_1(f)$ is tracked by the partial 
function $\overline{S_1(f)} = \big(S_1(1_a), S_1(f)\big)$. This functor $F^p $ though, is not, in general, full.
If $\C A^p = \Asm\big(\B {\CM}^p(\C C;S)\big)$ and $\Frg^{\C A^p}$ is
the corresponding forgetful functor, one can study the total computability model $\B {\CM}^t\big(\C A^p ;
\Frg^{\C A^p}\big)$ and the partial computability model $\B {\CM}^p\big(\C A^p ;
\Frg^{\C A^p}\big)$, and relate the latter to $\B {\CM}^p(\C C;S)$.



\section{Categories with a base of computability}
\label{sec: base}


\begin{definition}\label{def: base}
A base of 
computability\footnote{Having defined this notion independently, it was a nice surprise, 
and a posteriori not that surprising, to find out that a base of computability is Rosolini's notion of dominion,
defined in~\cite{Ro86}, section 2.1.} for $\C C$
is a family $B = (B(a))_{a \in C_0}$, where $B(a)$ is a subclass of the class $C_1(\eto a)$ of monos with codomain $a$, 
for every $a \in C_0$,
such that the following conditions are satisfied:\\[1mm]
$(\Base_1)$ For every $a \in C_0$, we have that $1_a \in B(a)$.\\[1mm]
$(\Base_2)$ For every $i \colon s \mto a \in B(a)$, for every $b \in C_0$, for every $f \colon s \to b$ 
and for every $j \colon t
\mto b \in B(b)$ a pullback $s \times_b t$ exists and $i \circ i{'} \colon s \times_b t \mto a \in B(a)$
\begin{center}
\begin{tikzpicture}

\node (E) at (0,0) {$s \times_b t$};
\node[right=of E] (F) {$t$};
\node[below=of E] (S) {$s$};
\node[below=of F] (A) {$b$.};
\node[left=of S] (K) {$a$};

\draw[left hook ->] (F)--(A) node [midway,right] {$j$};
\draw[->] (E) to node [midway,above] {$ f{'} $} (F) ;
\draw[right hook ->] (E)--(S) node [midway,right] {$i{'}$};
\draw[->] (S) to node [midway,below] {$f$ } (A) ;
\draw[left hook ->] (S)--(K) node [midway,below] {$i$};
\draw[blue,->] (E) to node [midway,above] {} (K) ;

\end{tikzpicture}
\end{center}
A cobase of computability for $\C C$ is a family $C = (C(a))_{a \in C_0}$, where $C(a)$ is a subclass of the 
class $C_1(a \eto)$ of monos with domain $a$ such that $(\Base_1)$ and the dual of $(\Base_2)$ are satisfied.
\end{definition}


Notice that the interesction 
of two bases 
need not be a base, if different pullbacks are considered in them.   

\begin{proposition}\label{prp: base1}
\normalfont(i) 
\itshape If $I(a) = \{1_a \colon a \mto a\}$, for every $a \in C_0$, then $I = (I(a))_{a \in C_0}$ is a base and a cobase
of computability for $\C C$.\\[1mm]
\normalfont(ii) 
\itshape If $\Iso(a) = \{i \in C_1(\eto a) \mid i \ \mbox{is an iso}\}$, for every $a \in C_0$, then 
$\Iso = (\Iso(a))_{a \in C_0}$ is a base of computability for $\C C$.\\[1mm]
\normalfont(iii) 
\itshape If $\C C$ has pullbacks, and $B(a) = C_1(\eto a)$, for every $a \in C_0$, then $B^{\eto} = (B(a))_{a \in C_0}$ is a base 
of computability for $\C C$.\\[1mm]
 \normalfont(iv) 
\itshape If $\C C$ has pushouts, and $C(a) = C_1(a \eto)$, for every $a \in C_0$, then $C^{\hookleftarrow}
= (C(a))_{a \in C_0}$ is a cobase of computability for $\C C$.
\end{proposition}

\begin{proof}
(i) We fix $a \in C_0$, $b \in C_0$ and $f \colon a \to b$. The only element of $I(b)$ is $1_b \colon b \mto b$, and 
the following is a pullback square with $1_a  \circ 1_a = 1_a \in I(a)$
\begin{center}
\begin{tikzpicture}

\node (E) at (0,0) {$a$};
\node[right=of E] (F) {$b$};
\node[below=of E] (S) {$a$};
\node[right=of S] (A) {$b$.};
\node[left=of S] (K) {$a$};

\draw[left hook ->] (F)--(A) node [midway,right] {$1_b$};
\draw[->] (E) to node [midway,above] {$ f $} (F) ;
\draw[right hook ->] (E)--(S) node [midway,right] {$1_a$};
\draw[->] (S) to node [midway,below] {$f$} (A) ;
\draw[left hook ->] (S)--(K) node [midway,below] {$1_a$};
\draw[blue,->] (E) to node [midway,above] {} (K) ;

\end{tikzpicture}
\end{center}
(ii) In this case $a = s \times_b t$. Cases (iii) and (iv) are trivial.
\end{proof}

%
%

\begin{definition}\label{def: basemodel}
If $B = (B(a))_{a \in C_0}$ is a base of computability for $\C C$ and 
$S \colon \C C \to \Set$ preserves pullbacks, the computability model $\B {\CM}^B(\C C ; S)$ over $\C C$
and $S$ with respect to $B$ is the pair
$$\B {\CM}^B(\C C ; S) = \big(\big(S_0(a)\big)_{a \in C_0}, \big(S[a, b]\big)_{(a, b) \in C_0 \times C_0}\big),$$
$$S[a, b] = \{S_1(i, f) \mid i \in B(a) \ \& \ f \in C_1(\dom(i), b)\}.$$
If $C = (C(a))_{a \in C_0}$ is a cobase of computability for $\C C$ and $T \colon \C C^{\op} \to \Set$  
sends pushouts to pullbacks, the computability model $\B {\CM}^B(\C C ; T)$ over $\C D$
and $T$ is defined dually i.e., 
$T[a, b] = \{T_1(i, f) \mid i \in C(a) \ \& \ f \in C_1(b, \cod(i))\}$.
\end{definition}

If $B(a)$ is a set, for every $a \in C_0$, and $\C C$ is locally small, then $\B {\CM}^B(\C C ; S)$ 
is locally small. 
By Proposition~\ref{prp: base1}, if $B = I$, then $\B {\CM}^I(\C C ; S) = \B {\CM}^t(\C C ; S)$. Actually, any functor 
preserves the pullbacks related to the base $I$. If 
$B = B^{\eto}$, then $\B {\CM}^{B^{\eto}}(\C C ; S) = \B {\CM}^p(\C C ; S)$.
Consequently, the results shown in the previous two sections can be seen as special cases of the corresponding 
results for categories with a base of computability.

%
%

\section{Concluding comments}
\label{sec: concl}

The notions and results presented here form the very first steps in the study of canonical computability models 
over categories and $\Set$-valued functors on them. 
Next we include some topics that could be investigated in the future.

To study more examples of computability models over categories with a given base of computability coming from
Rosolini's theory of dominions.
To study the computability models 
over a locally small category with pullbacks and the pullback preserving, covariant representable functors $\Hom(a, -)$.
To find interesting conditions on a computability model $\B C$ such that its category of assemblies $\Asm(\B C)$ has a base 
of computability $B$, and relate $\B C$ with $\B \CM^B\big(\Asm(\B C) ; \Frg^{\B C}\big)$.
To study the computability models over cartesian closed categories, as these are computability models over
a \textit{type world}, in the sense of ~\cite{LN15}, p.~66. To relate the various  
constructions in categories to their canonical computability models. 
To examine the impact of a Grothendieck topology on $\C C$ to some computability model over $\C C$.
To generalise the notions of a computability model and of a simulation between them so that the datatypes $C(\tau)$ are
not necesserily in $\Set$, but elements of some topos $\C S$, and then study $\C S$-computability models over 
a category $\C C$ and an $\C S$-valued functor on $\C C$. But most importantly, to relate the theory of
$\Set$-computability models over categories, introduced here, to concrete case studies from the theory of 
Higher-Order Computability, as this is developed in~\cite{LN15}. 

\end{document}